\documentclass[11pt]{amsart}
\usepackage{amssymb}
\usepackage{graphicx}
\usepackage{xcolor} 
\usepackage{tensor}
\usepackage{fullpage} 
\usepackage{tikz}
\usepackage{amsmath,amsthm}
\usepackage{verbatim}
\usepackage{hyperref}
\usepackage{array} 
\usepackage{enumitem}

\setlist[enumerate]{leftmargin=1.8em}
\setlist[itemize]{leftmargin=1.8em}
\usepackage{seqsplit,mathtools}

\definecolor{green}{rgb}{0,0.8,0} 

\usepackage{xcolor}

\newtheorem{thm}{Theorem}[section]  
\newtheorem{prop}{Proposition}[section]  
  
\newtheorem{lem}{Lemma}[section]  
  
\newtheorem{corollary}{Corollary}[section]

\newcommand{\R}{{\mathbb R}}

\def\cC{{\mathcal C}}

\numberwithin{equation}{section}

\newcommand{\nrm}[1]{\Vert#1\Vert}

\newcommand{\nnrm}[1]{{\vert\kern-0.25ex\vert\kern-0.25ex\vert #1 
		\vert\kern-0.25ex\vert\kern-0.25ex\vert}}

\newcommand{\ep}{\varepsilon}

\def\d{\partial}
\newcommand{\nb}{\nabla}

\def\div{\mathop{\rm div}\nolimits}

\newcommand{\with}{\quad\hbox{with}\quad}
\newcommand{\andf}{\quad\hbox{and}\quad}

\newcommand{\alp}{\alpha}

\newcommand{\tht}{\theta}

\newcommand{\omg}{\omega}

\newcommand{\bfe}{{\bf e}}

\newcommand{\bfu}{{\bf u}}

\newcommand{\bfomg}{{\boldsymbol{\omega}}}

\def\rd{\partial}
\def\bbR{\mathbb{R}}

\begin{document}

	\title[Euler equations in endpoint Sobolev space]{Global well-posedness for the incompressible Euler equations in an endpoint Sobolev space}
	 
	\date{\today}
	
	  \begin{NoHyper}
	\renewcommand{\thefootnote}{\fnsymbol{footnote}}
	\renewcommand{\thefootnote}{\arabic{footnote}}
	  \end{NoHyper}

\author{Rapha\"el Danchin}
\address{$^*$Univ Paris Est Creteil, Univ Gustave Eiffel, CNRS, LAMA UMR8050, F-94010 Creteil, France.}
\email{danchin@u-pec.fr}

\author{In-Jee Jeong}
\address{School of Mathematics, Korea Institute for Advanced Study, 85 Hoegi-ro, Seoul 02445, Republic of Korea.}
\email{ijeong@kias.re.kr} 

\begin{abstract} 
	We consider the initial value problem for the vorticity equation in the endpoint critical Sobolev space $W^{d,1}(\mathbb{R}^{d})$ for $d = 2, 3$. In two dimensions, we prove global propagation of the $W^{2,1}(\mathbb{R}^{2})$ regularity of the vorticity. In three dimensions, for axisymmetric flows without swirl, we propagate $W^{3,1}(\mathbb{R}^{3})$ regularity of the vorticity for all times. These are in stark contrast to existing strong ill-posedness results in critical Sobolev spaces $W^{d/p,p}(\mathbb{R}^{d})$ for all $1 < p < \infty$, which were based on axisymmetric flows without swirl when $d = 3$.  
\end{abstract}

\maketitle 

\section{Introduction}

We consider the initial value problem for the vorticity equation of incompressible Euler flows in $\bbR^d$ with $d = 2, 3.$
In $\bbR^3,$
it takes the form \begin{equation}\label{eq:V3}
		\rd_t \omg + u\cdot \nb \omg = \omg \cdot \nb u 
\end{equation}  with $\omg(t,\cdot): \bbR^3\to\bbR^3$ and, in $\bbR^2,$ it reads
 \begin{equation}\label{eq:V2}\d_t\omega +u\cdot\nabla \omega=0\end{equation}
 with scalar vorticity $\omg(t,\cdot): \bbR^{2} \to \bbR$.
\medbreak
 In both cases, $u$ is determined from $\omg$ at each moment of time by the Biot--Savart law $u = \nb\times\Delta^{-1}\omg. $ The equations are readily obtained by taking curl of the velocity equation \begin{equation} \label{eq:euler}
	\left\{	\begin{aligned} 		\d_tu+ u\cdot\nabla u+\nabla P=0 , & \\
	\div u=0.&
		\end{aligned}
	\right. \end{equation}


\subsection{Local and global well-posedness}

Since the pioneering papers by Lichtenstein  \cite{lichtenstein} one century ago that establish the existence of classical solutions, 
a huge number of mathematical works  addressed  the  local-in-time  well-posedness  of the initial value problem of the incompressible Euler equations. The vorticity equations \eqref{eq:V3} and \eqref{eq:V2} are locally well-posed in many reasonable ``subcritical'' Banach spaces, that are (locally) compactly embedded in the set $C^{0,1}_\sigma$ of  Lipschitz divergence-free vector fields. This includes in particular all   H\"older spaces  $C^{k,\alpha}$ with $k\geq 0$ and $\alpha\in(0,1)$ \cite{G,Ch} and, 
  if $1 < p < \infty,$ $1 \leq r < \infty$  and $s > d/p,$ Sobolev spaces $W^{s,p}$  \cite{KP1,KP2}, Besov spaces $B^s_{p,r}$ 
  \cite{C,Vishik,Vishik2,Zhou} and Triebel--Lizorkin spaces   $F^s_{p,r}$ \cite{CMZ}. (Even in subcritical regime, there may be ill-posedness when $r = \infty$; \cite{ChSh,JP,MY}.)

Well-posedness holds true as well in certain ``critical'' spaces, namely those embeds in $C^{0,1}$ continuously but not locally compactly. This is the case for Besov spaces with summability index 1, $\dot B^{d/p}_{p,1}$, including the endpoints $p=1, \infty$ (see \cite{PP,PP2}). 
    At the other end of the chain,  well-posedness has been shown to hold  for \eqref{eq:euler} 
    in  the smallest critical Triebel--Lizorkin space 
    $F^{d+1}_{1,\infty},$ which is a subspace of $W^{d+1,1}$ (see \cite{HP}).

   In dimension $d=2,$   the aforementioned results become global-in-time. This remarkable property, first pointed out
   by W. Wolibner  \cite{W} in 1933 for smooth data, 
   is a consequence of transport nature of $\omg$ in \eqref{eq:V2}.   In subcritical functional setting, the global existence can be seen as a consequence of the celebrated Beale--Kato--Majda criterion \cite{BKM}, which stipulates that there is blow-up at finite (positive) time $T$ if and only if 
\begin{equation}\label{eq:BKM} \int_0^T\| \omega (t) \|_{L^\infty}\,dt=\infty.\end{equation}
This holds in any dimension, and since in two dimensions \eqref{eq:V2} ensures that $\|\omega(t)\|_{L^\infty}$ is time independent, \eqref{eq:BKM} cannot occur. The  Beale--Kato--Majda criterion is based on a logarithmic interpolation inequality which is no longer valid in the borderline cases. Nevertheless,  M. Vishik proved global regularity for vorticity in $B^{2/p}_{p,1}$ with $p<\infty$ \cite{Vishik} (see the work \cite{HK} by T. Hmidi and S. Keraani for the case $p=\infty$).  
    
\subsection{Ill-posedness in critical Sobolev spaces} In light of the discussion above, the Sobolev spaces $W^{d/p,p}(\R^d)$ for $1 \le p \le \infty$, in terms of the vorticity, will be termed \textit{critical}. They have the same scaling invariance as $C^{0,1}$ of the velocity. 

The question of (local) well-posedness in these critical Sobolev spaces remained open until the breakthrough by J. Bourgain and D. Li in 2015 (\cite{BL1,BL3}). They obtained \textit{strong ill-posedness} of \eqref{eq:V2} and \eqref{eq:V3} in $W^{d/p,p}(\bbR^d)$ for all $1 < p < \infty$ and $d = 2, 3$, in the sense that there exist data in $W^{d/p,p}(\bbR^d)$ (with sufficient decay at infinity) such that there are \textit{no} corresponding solutions in $L^{\infty}(0,t;W^{d/p,p})$ for any $t>0$. Since $1 < p < \infty$, this gives strong ill-posedness of the velocity equation \eqref{eq:euler} in $W^{d/p +1 ,p}(\R^d)$. We shall discuss this work and related ones in more detail below, after stating our main results. 

The endpoint case $p = \infty$ is special; in two dimensions, the celebrated Yudovich theorem in 1963 \cite{Yudo} gives global in time existence and uniqueness of \eqref{eq:V2} with $\omg_{0} \in L^\infty \cap L^1(\bbR^2)$. On the other hand, in three dimensions, the right-hand side of \eqref{eq:V3} is responsible for strong ill-posedness with $\omg_{0} \in L^\infty \cap L^1(\bbR^3)$ (\cite{EM,JK}). In both dimensions, \eqref{eq:euler} is strongly ill-posedness with $C^{1,0}$ or $C^{0,1}$ velocity (\cite{BL2,EM,EJ2,JY}).

\subsection{The endpoint space \texorpdfstring{$W^{d,1}$}{TEXT}}

The other endpoint $p = 1$ was not discussed in the aforementioned works. Among the family of critical Sobolev spaces $W^{d/p , p}(\R^d)$, the space $W^{d,1}(\R^d)$ is distinguished since it is the only one for which if the vorticity is in it, then 
the velocity is in  $C^{0,1}.$ This already means that the strategy of Bourgain--Li cannot be used to give ill-posedness, because this was based on carefully choosing initial vorticities belonging to critical Sobolev spaces yet having velocity unbounded in $C^{0,1}$. On the other hand, well-posedness is not necessarily easier in $W^{d,1}$ compared to other critical spaces, since in this case one needs to control up to $d$ spatial derivatives of the vorticity. 

It seems that the issue of local well-posedness in $W^{d,1}$ has been open, until a very recent work of Cozzi and Harrison \cite{CH} which proved in $d = 2$ that the regularity $W^{2,1}$ of the vorticity is propagated locally in time. As said before,  it is expected that global well-posedness is true in dimension two, whenever local well-posedness 
  holds true.  Our first aim is to establish that, indeed the 
   $W^{2,1}$ regularity is propagated for all time in dimension two. 
   We will  then show that it is also the case of the    $W^{3,1}$ regularity in dimension three
   if we restrict ourselves to  axisymmetric solutions  without swirl, see the definition below.  

\subsection{Main results}



Let us first state our global existence result in two dimensions.
\begin{thm}\label{thm:d=2} 
Let the initial vorticity $\omega_0$ belong to $W^{2,1}(\R^2).$ 
Then, $\omega_0$ is in $\dot B^0_{\infty,1}(\R^2),$ 
$\nabla\omega_0$ is in the Lorentz space $L^{2,1}(\R^2)$ and 
\eqref{eq:V2} admits a unique global solution $\omega$  in $\cC(\R_+;W^{2,1}(\R^2))$ satisfying in 
 addition  
 $$\omega\in\cC(\R_+;\dot B^0_{\infty,1}(\R^2))\andf  \nabla\omega\in  \cC(\R_+; L^{2,1}(\R^2)).$$ 
 Furthermore,  there exists a universal constant $C$ such that for all $t\in\R_+,$ we have
$$\displaylines{\|\omega(t)\|_{\dot B^0_{\infty,1}}\leq C\|\omega_0\|_{\dot B^0_{\infty,1}}e^{Ct \|\omega_0\|_{\dot B^0_{\infty,1}}},
\quad\|\nabla^2\omega(t)\|_{L^1}\leq \|\nabla^2\omega_0\|_{L^1} \exp 
\biggl(\frac{e^{Ct \|\omega_0\|_{\dot B^0_{\infty,1}}}-1}{\|\omega_0\|_{\dot B^0_{\infty,1}}}\biggr) \cr\andf
\|\nabla\omega(t)\|_{L^{2,1}}\leq \|\nabla\omega_0\|_{L^{2,1}}
 \exp \biggl(\frac{e^{Ct \|\omega_0\|_{\dot B^0_{\infty,1}}}-1}{\|\omega_0\|_{\dot B^0_{\infty,1}}}\biggr) \cdotp}$$
\end{thm}
Granted with the above result, it is then natural to address well-posedness for the three-dimensional Euler equations
supplemented with initial vorticity in $W^{3,1}(\R^3).$  Compared 
to the two-dimensional case, the key difficulty is the so-called stretching term in the right-hand side of \eqref{eq:V3}. 
It is well-known that this stretching term  is nicer if we consider \textit{axisymmetric flows without swirl}. In cylindrical coordinates,  these solutions are of the form 
$$u = u^{r}(t,r,z) \bfe^{r} + u^{z}(t,r,z) \bfe^{z},$$  so that the vorticity is 
  $$\omg = \omg^{\tht}(t,r,z)\bfe^{\tht}\with \omg^{\tht}(t,r,z)=\d_ru^{z}(t,r,z) -\d_zu^{r}(t,r,z).$$ 
   Denoting  $\nabla = (\partial_{r}, \partial_{z}),$  Equation \eqref{eq:V3} reduces to
		\begin{equation}\label{eq:axiEuler-wo}
		\d_t \omg^{\tht} + (u \cdot \nb) \omg^{\tht} = \frac{u^r}{r} \omg^{\tht}. \end{equation}
		Therefore,  $\alpha:=\omega^\theta/r$ is just transported by the flow:  		
	  \begin{equation}\label{eq:axiEuler-wo-rel}
			\d_t \alpha + u \cdot \nb \alpha = 0.	\end{equation}
As first observed,  by Ukhovski\`i and V.I. Yudovich \cite{UY}, this remarkable property ensures 
the global existence for smooth initial data with, in particular, $\alpha_0$ bounded. 
			
	In the present work, we establish that for axisymmetric data without swirl, 
	the critical Sobolev regularity for the vorticity is  preserved for all time:
\begin{thm}\label{thm:axi} Assume that $\omega_0$ belongs to $W^{3,1}(\R^3)$ and 
has the above axisymmetric structure.  
Then,  $\alpha_0$ and $\nabla\omega_0$ belong to the Lorentz space $L^{3,1},$  $\nabla^2\omega_0$ is in $L^{3/2,1},$
$\omega_0$ is in $\dot B^0_{\infty,1}$ and 
\eqref{eq:V3} admits a unique global solution $\omega$ with constant $\|\alpha(t)\|_{L^{3,1}}$ and $$\omega\in\cC(\R_+;\dot B^0_{\infty,1}), \quad
\nabla\omega \in \cC(\R_+;L^{3,1}),\quad
\nabla^2\omega\in \cC(\R_+;L^{3/2,1})\andf \nabla^3\omega\in\cC(\R_+;L^1).$$ 
{The quantities $\nrm{\omg(t)}_{\dot B^0_{\infty,1}}$ and $\nrm{\nb\omg(t)}_{L^{3,1}}$ grow at most exponentially in time, while $\nrm{\nb^2\omg(t)}_{L^{3/2,1}}$ and $\nrm{\nb^3\omg(t)}_{L^1},$ double exponentially in time.}
\end{thm}
We conclude this section by a few words about the strategy to achieve our main results. 
In both cases, the proof relies heavily on the fact that $W^{d,1}(\R^d)$ embeds continuously into the Besov space $B^0_{\infty,1}(
\R^d)$, and that, according to \cite{HK}, this latter regularity is propagated  for all time  in dimension two. 
 In the three-dimensional axisymmetric case without swirl,  it turns out that having 
  $\omega_0$  in $W^{3,1}$, not only implies that $\omega_0\in  B^0_{\infty,1}$
  but also that  $\omega_0/r$ lies in the Lorentz space $L^{3,1},$
  a key ingredient to propagate regularity for all time, see \cite{AHK,D}.

\subsection{Further remarks and open problems} The key idea, originating from Bourgain--Li's works \cite{BL1,BL2,BL3}, in the proof of ill-posedness in critical spaces is to choose initial vorticities which have unbounded Lipschitz norm of the corresponding velocity. Then, the vorticity gets ``stretched'' by the large velocity gradient, giving growth in the critical norm. 

In this growth scenario, the fact that the vorticity in two dimensions is just being transported significantly simplifies the analysis. For a similar reason, in Bourgain--Li \cite{BL3}, ill-posedness in $d = 3$ with vorticity in $W^{3/p,p}(\R^3)$ ($1<p<\infty$) was proved using axisymmetric flows without swirl \eqref{eq:axiEuler-wo}, by considering axisymmetric data whose solutions are behaving similarly as in two dimensions. (Actually, \cite{BL3} give details of the proof only in the case $p = 2$, but the general case of $1<p\le \infty$ can be rather easily proved following the simplified arguments of \cite{EJ1,JK,JK2}, which again use axisymmetric flows without swirl.) Therefore, Theorem \ref{thm:axi} clearly shows that these strategies for ill-posedness cannot work for the endpoint space $W^{3,1}(\bbR^{3})$.

The question of well/ill-posedness of \eqref{eq:V3} with \textit{general} $W^{3,1}(\bbR^3)$ vorticity is still open. For  $\omega_0$ in $W^{3,1}(\R^3),$ not necessarily axisymmetric, we still have $\omega_0\in\dot B^0_{\infty,1},$
$\nabla\omega_0\in L^{3,1}$ and $\nabla^2\omega_0\in L^{3/2,1}.$ 
Since the first property provides us with a control of the Lipschitz norm, it is rather easy to combine elementary 
estimates for the transport equation and to produce a local-in-time strong solution $\omega$ such that 
$$\omega\in\cC([0,T];\dot B^0_{\infty,1}),\quad \nabla\omega\in \cC([0,T[;L^{3,1})\andf 
\nabla^2\omega\in \cC([0,T[;L^{3/2,1}).$$
Using endpoint product estimates and the fact that the Biot--Savart law maps 
$L^1$ to $\dot B^0_{1,\infty}$ should  allow us to show the conservation of Besov regularity $\dot B^0_{1,\infty}$ of $\nabla^3\omega.$
At the same time, it is not clear at all that one can propagate the regularity $L^1$ of  $\nabla^3\omega.$
We plan to address this challenging  question in a future work. 

Let us also remark that the question of well-posedness in $W^{d+1,1}(\bbR^d)$ for the velocity seems to be open, even in two dimensions, which is not equivalent to the $W^{d,1}(\R^d)$ problem for the vorticity.

Now let us describe a few very recent developments on critical space ill-posedness.  Regarding ill-posedness in $W^{d/p,p}(\R^d)$ with $1<p<\infty$, it may seem from the discussion above that one can just require additionally $u_{0} \in C^{0,1}$ (or $C^{1,0}$) to ``restore'' well-posedness. However, this was ruled out by Jo and Kim in \cite{JoK}: there exist data $u_{0} \in W^{2/p+1,p} \cap C^{1,0}(\bbR^2)$ with $1<p<\infty$ without any local solutions in $L^{\infty}(0,t;W^{2/p+1,p})$. In this work, the choice of initial data is more subtle; one needs to arrange it in a way that the velocity \textit{instantaneously} becomes non-Lipschitz. Moreover, ill-posedness in the critical Besov $B^{0}_{\infty,q} $ with $1<q<\infty$ was finally settled by Shikh Khalil \cite{SK}. Regarding the axisymmetric vorticity equation without swirl \eqref{eq:axiEuler-wo}, Bang and Cheskidov showed sharpness of the Lorentz space assumption $\omg_{0}/r \in L^{3,1}$, in the sense that there is strong ill-posedness with $L^{3,q}$ with any $q> 1$ \cite{BC}, extending an earlier work \cite{JK}.

Furthermore, there are various strong ill-posedness statements available for supercritical spaces \cite{BT, CMO,Luo,Jeong,JMO}. There are certain supercritical spaces in two dimensions in which \eqref{eq:V2} is locally well-posed \cite{EMS}.

Lastly, we refer to \cite{BL1,El,SK} for extensive lists on works related to well-posedness of the Euler equations.

\subsection*{Organization of the paper}
The rest of the article is organized as follows.
In the next section, we establish the all-time propagation of Sobolev regularity in the two-dimensional case,
while  the three-dimensional axisymmetric case is treated in Section \ref{s:axi}. 
Some useful results involving Lorentz and Besov spaces are recalled in the Appendix,  as well as 
technical equivalences involving the norm of derivatives of axisymmetric solutions.

\section{The two-dimensional case}

Consider some initial vorticity $\omega_0$ in $W^{2,1}(\R^2).$  Then, according to  the work by Cozzi and Harrison
\cite{CH}, 
the vorticity formulation of the Euler equations \eqref{eq:V2} possesses a unique maximal 
solution $\omega\in\cC([0,T^*); W^{2,1}).$ As a by-product of the proof therein, we see 
that there exists a universal constant $c$ such that 
$$T^*\geq c\|\omega_0\|_{W^{2,1}}^{-1}.$$  
By standard argument, this implies that if $T^*<\infty$ then we must have 
$$\limsup \|\omega(t)\|_{W^{2,1}}=\infty\quad\hbox{for}\quad t\to T^*.$$ 
Hence,  in order to get $T^*=\infty,$  it suffices to establish the inequalities
that are stated in Theorem \ref{thm:d=2}  on the interval $[0,T^*).$ 
\medbreak
To proceed, we combine  Propositions  \ref{l:Lorentz1} and \ref{l:Lorentz2}  to get
\begin{equation}\label{eq:keyineq}
\|\omega\|_{\dot B^0_{\infty,1}} \approx \|\nabla \omega\|_{\dot B^{-1}_{\infty,1}}
\lesssim \|\nabla\omega\|_{L^{2,1}}
\lesssim \|\nabla^2\omega\|_{L^1}.\end{equation}
Now, since $u$ is divergence free and $\omega$ is just transported by the flow of $u,$ we have
$$\|\omega(t)\|_{L^1}=\|\omega_0\|_{L^1}.$$
Next, differentiating \eqref{eq:V2} once, we get 
\begin{equation}\label{eq:d1}
\d_t\nabla \omega+(u\cdot\nabla)\nabla \omega=-\nabla u\cdot\nabla\omega,\end{equation}
whence
$$\|\nabla\omega(t)\|_{L^1} 
 \leq \|\nabla\omega_0\|_{L^1} +\int_0^t\|\nabla u\|_{L^\infty}\|\nabla\omega\|_{L^1}\,d\tau$$
 and also, 
 \begin{equation}\label{eq:W21}
 \|\nabla\omega(t)\|_{L^{2,1}} 
 \leq \|\nabla\omega_0\|_{L^{2,1}} +\int_0^t\|\nabla u\|_{L^\infty}\|\nabla\omega\|_{L^{2,1}}\,d\tau.
 \end{equation} 
 Finally,  differentiating \eqref{eq:V2} twice yields for all $1\leq i,j\leq2$:
 \begin{equation}\label{eq:d2}
 (\d_t+u\cdot\nabla)\d^2_{ij}\omega =-\d^2_{ij} u\cdot\nabla \omega-\d_ju\cdot\nabla\d_i\omega-\d_iu\cdot\nabla\d_j\omega,
 \end{equation}
 whence using transport estimate in $L^1$ and H\"older inequality, 
 $$
 \|\nabla^2\omega(t)\|_{L^1}\leq \|\nabla^2\omega_0\|_{L^1}+C\int_0^t\|\nabla u\|_{L^\infty}\|\nabla^2\omega\|_{L^1}\,d\tau
 +\int_0^t\|\nabla^2 u\|_{L^2}\|\nabla \omega\|_{L^2}\,d\tau.$$
 Now, we observe that $\|\nabla^2 u\|_{L^2}\lesssim \|\nabla \omega\|_{L^2}$ (see Prop. \ref{l:BS}) and that
 $$ \|\nabla\omega\|_{L^2}^2=\int \nabla\omega \cdot\nabla \omega\,dx =-\int\omega\Delta\omega\,dx
 \leq \|\omega\|_{L^\infty}\|\Delta\omega\|_{L^1}.$$
 Hence, we conclude that 
  $$
 \|\nabla^2\omega(t)\|_{L^1}\leq \|\nabla^2\omega_0\|_{L^1}+C\int_0^t\|\nabla u\|_{L^\infty}\|\nabla^2\omega\|_{L^1}\,d\tau.$$
 Putting all these estimates together and combining with Gr\"onwall lemma, we end up with 
 \begin{equation}\label{eq:omega21}
\|\omega(t)\|_{W^{2,1}}\leq \|\omega_0\|_{W^{2,1}} \exp\biggl( C\int_0^t\|\nabla u\|_{L^\infty}\,d\tau \biggr)\cdotp\end{equation}
Since 
\begin{equation}\label{eq:B}
\|\nabla u\|_{L^\infty}\lesssim \|\nabla u\|_{\dot B^0_{\infty,1}}\lesssim \|\omega\|_{\dot B^0_{\infty,1}}\end{equation}
and,  according to Vishik's result \cite{Vishik}, 
 $$\|\omega(t)\|_{\dot B^0_{\infty,1}} \leq C \|\omega_0\|_{\dot B^0_{\infty,1}}\biggl(1+\int_0^t\|\nabla u\|_{L^\infty} d\tau \biggr),$$
  using \eqref{eq:B} followed by Gr\"onwall lemma, we conclude that
 $$ \|\omega(t)\|_{\dot B^0_{\infty,1}}\leq C\|\omega_0\|_{\dot B^0_{\infty,1}}e^{Ct \|\omega_0\|_{\dot B^0_{\infty,1}}}
 \quad\hbox{for all }\ t\in[0,T^*).$$
 Then, plugging this information in \eqref{eq:W21} and \eqref{eq:omega21} and using 
 again Gr\"onwall lemma, we get the double exponential growth stated in 
 Theorem \ref{thm:d=2}.


\section{The 3D axisymmetric case with no swirl}\label{s:axi}

This part is devoted to the proof of our second theorem. 
We mainly focus on the proof of global-in-time priori estimates. 
In order to get the rigorous existence statement, the fastest way is to smooth out the initial data by 
a family of radial nonnegative mollifiers. In this way, we get a uniform control of the $W^{3,1}$ norm 
and keep the axisymmetric structure of the data. Then, the a priori estimates that will be presented below ensure uniform estimates in the desired spaces, 
and it is expected that compactness results will allow us to pass to the limit. 
The difficulty however is the same as  in the two-dimensional case  studied in \cite{CH}:
 since the space $L^1$ is not stable by weak * compactness, it may happen  that the third order derivatives of 
 the vorticity  are only bounded measures rather than  nice $L^1$ functions.
This can be overcome by following faithfully the approach of \cite{CH} for the 2D case, and thus omitted.
\medbreak 
Throughout this section, all the norms are taken with respect to the Lebesgue measure on $\bbR^{3}$, unless otherwise specified. Furthermore, in the rest of this section, to avoid potential ambiguity in notation, we use boldface letters $\bfu, \bfomg$ to denote the  velocity and vorticity defined in Cartesian coordinates; namely, $\bfu(t,x) = u^{r}(t,r,z) \bfe^{r} + u^{z}(t,r,z) \bfe^{z}$ and $\bfomg(t,x) = \omg^{\tht}(t,r,z)\bfe^{\tht}.$ Furthermore, we denote $u(t,r,z) = (u^r(t,r,z),u^z(t,r,z))$. For functions defined in Cartesian coordinates, $\nb $ refers to the usual Cartesian gradient, while we abuse the same notation to mean $\nb = \nb_{r,z}$ when applied to functions defined in the $(r,z)$-plane. We frequently use the equivalences of derivatives in Cartesian and axisymmetric coordinates summarized in Section \ref{sec:equivalence}. 

 \bigbreak

 In the rest of this section, we focus on the proof of a priori estimates leading to Theorem \ref{thm:axi}. 
 The starting point is Proposition \ref{l:Lorentz1} which implies that $\nabla^2\bfomg_0\in L^{3/2,1},$ and 
 thus, due to a more standard embedding, $\nabla\bfomg_0\in L^{3,1}$  and, finally
 $\bfomg_0\in\dot B^0_{\infty,1},$ due to Proposition \ref{l:Lorentz2}. 
 Since  $\nabla\bfomg_0\in L^{3,1}$  implies that $\alpha_0\in L^{3,1}$ (see Lemma \ref{l:nablaomega})
 and  all Lorentz and Lebesgue norms of $\alpha$ are time independent due to \eqref{eq:axiEuler-wo-rel}, and we  have 
 for all $t>0,$
$$\|\alpha(t)\|_{L^{3,1}}=\|\alpha_0\|_{L^{3,1}}\lesssim \|\nabla \bfomg_0\|_{L^{3,1}}\lesssim \|\nabla^3 \bfomg_0\|_{L^1}<\infty.$$
As, in addition, the data is axisymmetric without swirl and   
$\bfomg_0\in \dot B^0_{\infty,1},$ leveraging the results of \cite{AHK,D} already ensures 
that \eqref{eq:V3} has a global unique solution $\omega\in \cC(\R_+;\dot B^0_{\infty,1})$ with constant
$\|\alpha(t)\|_{L^{3,1}}. $
 In the following steps, we successively establish the preservation of regularity
 $L^{3,1}$  of  $\nabla\bfomg,$   regularity
 $L^{3/2,1}$  of  $\nabla^2\bfomg$ and, finally,  regularity $L^{1}$  of  $\nabla^3\bfomg.$


	\subsection{Well-posedness with \texorpdfstring{$\nb\bfomg_0 \in L^{3,1}(\R^3)$}{TEXT}} In this section, we obtain global a priori estimates in the class $\nb\bfomg  \in L^{3,1}(\R^3)$ and $\bfomg \in L^{1}(\R^3)$ for solutions of  \eqref{eq:axiEuler-wo}.

	From $\nb\bfomg_0 \in L^{3,1}$, we have that $ \alpha_0 \in L^{3,1}$, which is preserved for all times by \eqref{eq:axiEuler-wo-rel}. In particular we have the global in time uniform bound  (see \cite{D}):
	\begin{equation}\label{eq:valpha}
					\nrm{r^{-1}u^r}_{L^{\infty}(\R^3)} \le C \nrm{\alpha}_{L^{3,1}(\R^3)}=C \nrm{\alpha_0}_{L^{3,1}(\R^3)}.
		\end{equation} 
		By \cite{D}, we have uniqueness and existence of solution to \eqref{eq:axiEuler-wo} satisfying 
		\begin{equation}\label{eq:omegainfty}
			\nrm{\omg^{\tht}(t,\cdot)}_{L^{\infty}} \le 	\nrm{\omg^{\tht}_0}_{L^{\infty}}  \exp(C\nrm{\alpha_0}_{L^{3,1}} t). 	
		\end{equation} 
	To propagate 
	  $\nb_{r,z} \omg^\tht \in L^{3,1},$  we apply  $\nb^\perp = \nb^\perp_{r,z}$ to  \eqref{eq:axiEuler-wo}, and get
 \begin{equation*}
\d_t \nb^\perp\omg^\tht + u\cdot\nb \nb^\perp\omg^\tht = \frac{u^r}{r} \nb^\perp\omg^\tht + \nb^\perp\left( \frac{u^r}{r} \right)\omg^{\tht}.\end{equation*} 
The first term of the right-hand side can be estimated in $L^{3,1}$ as follows 
	using \eqref{eq:valpha}:
			\begin{equation}
			\|r^{-1}u^r  \nb^\perp\omg^\tht\|_{L^{3,1}}\leq \|r^{-1}u^r\|_{L^\infty}\|\nb^\perp\omg^\tht\|_{L^{3,1}}
			\lesssim  \nrm{\alpha_0}_{L^{3,1}}\|\nb^\perp\omg^\tht\|_{L^{3,1}}.\end{equation}
		To bound the last term, we observe that
		$$\nb_r( r^{-1}u^r )\omg^{\tht}=- r^{-1}u^r\,\alpha+\nb_ru^r \,\alpha\andf
		\nb_z(r^{-1}u^r)\omg^{\tht}=\nb_zu^r\, \alpha.$$
		Therefore, 	using  Lemma \ref{l:nablau} and Proposition \ref{l:BS} yields		
\begin{equation*}			
\| \nb^\perp( r^{-1}u^r )\omg^{\tht}\|_{L^{3,1}}\leq  (\|\nb^\perp u^r\|_{L^{\infty}} + \nrm{r^{-1}u^r}_{L^{\infty}} )\|\alp \|_{L^{3,1}} \lesssim \|\nabla \bfu\|_{L^{\infty}}\|\alp_0 \|_{L^{3,1}} .
\end{equation*}
Using the chain of inequalities following from Proposition \ref{l:Lorentz2} 
$$ \nrm{ \nb \bfu}_{L^\infty}\lesssim  \|\nb \bfu\|_{\dot B^0_{\infty,1}} \lesssim \|\nb^2 \bfu\|_{L^{3,1}}  \lesssim \nrm{\nb\bfomg}_{L^{3,1}} \lesssim \|\nb^\perp\omg^\tht\|_{L^{3,1}} + \|\alp_0 \|_{L^{3,1}},$$ we get 
\begin{equation*}
		\frac{d}{dt} \nrm{ \nb^\perp \omg^\tht  }_{L^{3,1}} \lesssim \|\alp_0 \|_{L^{3,1}} \left( \nrm{ \nb^\perp  \omg^\tht }_{L^{3,1}} + \|\alp_0 \|_{L^{3,1}}\right)\cdotp
\end{equation*}
Hence, by Gr\"onwall  lemma (and then using the above chain of inequalities), we conclude that
 \begin{equation}\label{eq:nablau}
		\nrm{ \nb u(t,\cdot)}_{L^\infty} + \nrm{ \nb^2 \bfu(t,\cdot)}_{L^{3,1}}  +  \nrm{\nb\omg^{\tht}(t,\cdot)}_{L^{3,1}} \le C_0 \exp(C_0t),
		\end{equation} 
 with $C_0$ depending only on suitable critical norms of $\bfomg_0.$ 
	
	\subsection{Well-posedness with \texorpdfstring{$\nb^2\bfomg_0 \in L^{3/2,1}(\R^3)$}{TEXT}}
	We take advantage of Lemma \ref{l:nablaomega}: it suffices to
	 estimate  $\nb^\perp \alpha$ and $\nabla^2_{r,z}\omg^\tht$ in $L^{3/2,1}.$ To do this, we observe that  \begin{equation*}
					\d_t \nb^\perp \alpha + u\cdot\nb \nb^\perp\alpha =  \nb u \cdot \nb^\perp \alpha
		\end{equation*} which gives 
		\begin{equation*}
			\frac{d}{dt} \nrm{ \nb^\perp \alpha }_{L^{3/2,1}} \le C \nrm{\nb u }_{L^\infty} \nrm{ \nb^\perp \alpha }_{L^{3/2,1}}. 
		\end{equation*}
		 Recalling \eqref{eq:nablau}, this shows that $\nrm{ \nb^\perp \alpha }_{L^{3/2,1}}$  grows at most double exponentially in time. 
	\medbreak
	Then, for a partial derivative $\d$ ($\d_z$ or $\d_r$),  we have\begin{multline*}
	(\d_t+u\cdot\nabla) \d \nb^\perp\omg^{\tht} = r^{-1}u^r \d \nb^\perp\omg^{\tht} - \d u \cdot \nb   \nb^\perp\omg^{\tht} \\+ \d(r^{-1}u^r)  \nb^\perp\omg^{\tht} + \d  \omg^{\tht} \nb^\perp (r^{-1}u^r) +  {\omg^{\tht} \d \nb^\perp( r^{-1} u^{r})} . 
	\end{multline*}
	 The first two terms in the right-hand side are bounded in $L^{3/2,1}$ by $\nrm{\nb u}_{L^{\infty}} \nrm{ \nb^2 \omg^{\tht} }_{L^{3/2,1}}$.  
	\medbreak
	The next two terms are bounded as follows: \begin{align*}
					\nrm{ \d(r^{-1}u^r)  \nb^\perp\omg^{\tht} }_{L^{3/2,1}} + \nrm{ \d  \omg^{\tht} \nb^\perp (r^{-1}u^r) }_{L^{3/2,1}}    &\le C\nrm{ \nb (r^{-1}u^r) }_{L^{3,1 }} \nrm{ \nb^\perp\omg^{\tht} }_{L^{3,1}}\\
					  &\le  C \left(  \nrm{ \nb^\perp\omg^{\tht} }_{L^{3,1}} + \nrm{\alpha}_{L^{3,1}} \right) \nrm{ \nb^\perp\omg^{\tht} }_{L^{3,1}} . 
	\end{align*}	
The last term is slightly trickier. As an example, we consider the case when we have two $r$ derivatives acting on $r^{-1} u^{r}.$
Then, we use the fact that
 \begin{equation*}\d^2_{rr} ( r^{-1} u^{r}) \omg^{\tht} =  \d^2_{rr}u^{r} \frac{\omg^{\tht}}{r} - 2 \d_{r}\left( \frac{u^r}{r} \right) \frac{\omg^{\tht}}{r}\cdotp
		\end{equation*} 
We thus have 
		\begin{align*}
		\|\d^2_{rr} ( r^{-1} u^{r}) \omg^{\tht} \|_{L^{3/2,1}}&\lesssim
		\bigl(	\nrm{ \d^2_{rr}u^{r} }_{L^{3}} + \nrm{ \d_{r}( r^{-1} u^{r} )}_{L^{3,1}}\bigr)
		\|\alpha\|_{L^{3,1}}\\ &\lesssim \bigl( \nrm{ \nb^\perp\omg^{\tht} }_{L^{3,1}} +\nrm{ \alpha }_{L^{3,1}} \bigr)
		\|\alpha_0\|_{L^{3,1}}.
		 \end{align*}
		 Finally, 
		$$	\|\d^2_{zz} ( r^{-1} u^{r}) \omg^{\tht} \|_{L^{3/2,1}}+\|\d^2_{rz} ( r^{-1} u^{r}) \omg^{\tht} \|_{L^{3/2,1}}
				\lesssim \|\nabla^3 \bfu\|_{L^{3,1}}\|\omega^\theta\|_{L^\infty}
		\lesssim \|\nabla^2\omg^\tht\|_{L^{3,1}}\|\omega^\theta\|_{L^\infty}.$$
	Putting together all these inequalities and recalling \eqref{eq:nablau} and exponential in time bound
	\eqref{eq:omegainfty} for $\|\omega^\theta\|_{L^\infty}$, \begin{equation*}
			\frac{d}{dt} \|\nb^2\omega^\theta\|_{L^{3/2,1}} \le C_0\exp(C_0t) \|\nb^2\omega^\theta\|_{L^{3/2,1}} + C_0\exp(C_0t). 
	\end{equation*} Using Gr\"onwall lemma and the estimates of the previous parts gives double exponential in time upper bound for $\nb^2\bfomg(t) \in L^{3/2,1}(\R^3).$

	
	\subsection{Well-posedness for \texorpdfstring{$\bfomg_0 \in W^{3,1}(\R^3)$}{TEXT}}
	
	This is of course the most tricky part. 
	To simplify the notation, we set 
\begin{equation}\label{def:alphav}D_t:= \d_t+u\cdot\nabla,\quad \alpha:=r^{-1}\omega^\theta\quad\hbox{and}\quad v:=r^{-1}u^r.\end{equation}
Applying Corollary \ref{cor:Lp} to what we have obtained so far, we know how to control the following quantities in terms of the data: 
$$\displaylines{
\|\nabla_{r,z} u(t)\|_{L^\infty},\quad \|v(t)\|_{L^\infty},\cr
\|\nabla^2_{r,z}u(t)\|_{L^{3,1}},\quad \|\nabla_{r,z}v(t)\|_{L^{3,1}},\quad \|r^{-1}\d_r u^z(t)\|_{L^{3,1}}, \quad 
\|\nabla_{r,z}\omega^\theta(t)\|_{L^{3,1}},
\quad \|\alpha(t)\|_{L^{3,1}},\cr
\|\nabla^3_{r,z}  {u(t)}\|_{L^{3/2,1}},\quad  \|\nabla^2_{r,z}v(t)\|_{L^{3/2,1}},\quad
 \|r^{-1}\partial_r v(t)\|_{L^{3/2,1}},\quad \|\nabla_{r,z}(r^{-1}\partial_r u^z)(t)\|_{L^{3/2,1}},\cr
 \|\nb_{r,z}^2 \omega^\theta(t)\|_{L^{3/2,1}},\quad  \nrm{ \nb_{r,z} \alpha(t)}_{L^{3/2,1}}.}$$
 {The last six quantities may grow up to double exponentially in time, all others only single exponentially in time.}

According to Lemma \ref{l:nablaomega}, in order  to control 
 $\|\bfomg(t)\|_{\dot W^{3,1}},$
it suffices to bound  $r^{-1}\d_r\alpha,$ 
$\nabla^2_{r,z}\alpha$ and $\nabla^3_{r,z}\omega^\theta$ in $L^1,$ in terms of $\|\bfomg_0\|_{W^{3,1}}.$
The trickiest part is to get control of $\nabla^3_{r,z}\omega^\theta$ in $L^1$ (that is,  $\d^3_{zzz}\omega^\theta,$  $\d^3_{rzz}\omega^\theta,$  $\d^3_{rrz}\omega^\theta$ and  $\d^3_{rrr}\omega^\theta$). 

We shall use repeatedly that $D_t$ acts on scalar functions depending only on $r$ and $z,$ as the convective
derivative corresponding to the divergence-free vector-field $u.$

\subsubsection*{Bounding  $r^{-1}\d_r\alpha$ in $L^1$} We observe that 
	$$	D_t(r^{-1}\d_r\alpha)=-\bigl(v+\d_r u^r\bigr)\bigl(r^{-1}\d_r\alpha\bigr)
	-\d_ru^z \bigl(r^{-1}\d_z\alpha\bigr)\cdotp$$
For the last term, we can use that
$$\|r^{-1}\d_r u^z\d_z\alp\|_{L^1}\leq \|r^{-1}\d_r u^z\|_{L^3}\|\d_z\alpha\|_{L^{3/2}}.$$
	Hence 
	\begin{align*}
	\frac{d}{dt} \|r^{-1}\d_r\alpha(t)\|_{L^1}\leq (\|v\|_{L^\infty}+\|\d_r u^r\|_{L^\infty})\|r^{-1}\d_r\alpha\|_{L^1} + \|r^{-1}\d_r u^z\|_{L^3}\|\d_z\alpha\|_{L^{3/2}}. \end{align*}
Recalling exponential upper bounds for $\|v\|_{L^\infty}+\|\d_r u^r\|_{L^\infty}$ and  $\|r^{-1}\d_r u^z\|_{L^3}\|\d_z\alpha\|_{L^{3/2}}$, we deduce double exponential upper bound for $\|r^{-1}\d_r\alpha(t)\|_{L^1}$. 
 
\subsubsection*{Bounding  $\nabla^2_{r,z}\alpha$ in $L^1$}
Differentiating the equation of $\alpha$ twice, we have for $\sigma,\rho\in\{r,z\},$ 
$$
D_t\d^2_{\sigma\rho}\alpha = \d_\sigma u\cdot\nabla\d_\rho\alpha+\d_\rho u\cdot\nabla\d_\sigma\alpha
+\d^2_{\sigma\rho}u\cdot\nabla\alpha.$$ We may estimate $\nrm{\d^2_{\sigma\rho}u\cdot\nabla\alpha}_{L^{1}} \le \|\nabla^2 u\|_{L^3}\|\nabla \alpha\|_{L^{3/2}}$, which grows at most double exponentially. This gives a double exponential control of $\|\nabla^2\alpha(t)\|_{L^1}.$ 

\subsubsection*{Bounding  $\d^3_{zzz}\omega^\theta$} We differentiate \eqref{eq:axiEuler-wo} three times with respect to $z,$ getting 
\begin{multline*}
D_t \d^3_{zzz}\omega^\theta = \alpha \d^3_{zzz}u^r+3 {\d_z\alpha} \d^2_{zz}u^r+ 3 \d^2_{zz}\alpha\d_{z}u^r + v\d^3_{zzz}\omega^\theta\\-3\d_zu\cdot\nabla\d^2_{zz}\omega^\theta
-3\d^2_{zz}u\cdot\nabla\d_{z}\omega^\theta-\d^3_{zzz}u\cdot\nabla\omega^\theta.
\end{multline*}
Each term of the right-hand side can be  bounded in $L^1$ as follows: 
\begin{align*}
\|\alpha\d^3_{zzz}u^r\|_{L^1}&\lesssim  \|\alpha\|_{L^3}\|\d^3_{zzz}u^r\|_{L^{3/2}}\\
\| \d_z\alpha\d^2_{zz}u^r\|_{L^1}&\lesssim  \|\d_z\alpha\|_{L^{3/2}}\| \d^2_{zz}u^r\|_{L^3}\\
 \| \d^2_{zz}\alpha\d_{z}u^r \|_{L^1}&\lesssim \|\d^2_{zz}\alpha\|_{L^1}  \| \d_{z}u^r\|_{L^\infty}\\
 \| v \d^3_{zzz}\omega\|_{L^1}&\lesssim  \| v\|_{L^\infty}\| \d^3_{zzz}\omega^\theta\|_{L^1}\\
 \|\d_zu\cdot\nabla\d^2_{zz}\omega^\theta\|_{L^{1}}  &\lesssim  \|\d_zu\|_{L^\infty}\|\nabla\d^2_{zz}\omega^\theta\|_{L^1}\\
\|\d^2_{zz}u\cdot\nabla\d_{z}\omega^\theta\|_{L^1}&\lesssim \|\d^2_{zz}u\|_{L^3}\|\nabla\d_{z}\omega^\theta\|_{L^{3/2}}\\
\|\d^3_{zzz}u\cdot\nabla\omega^\theta\|_{L^1}&\lesssim \|\d^3_{zzz}u\|_{L^{3/2}}
\|\nabla\omega^\theta\|_{L^3}.
\end{align*}
Observe that in the right-hand sides of these inequalities, all the quantities multiplied with $\nrm{\nb^3\omg^{\tht}}_{L^{1}}$ grow only up to single exponentially in time. This gives \begin{equation*}
		\frac{d}{dt} \nrm{ \d^3_{zzz}\omega^\theta }_{L^{1}} \le C_0 \exp(C_0t) \nrm{ \nb^{3}\omega^\theta }_{L^{1}}  + C_0 \exp( C_0 \exp (C_0t)), 
\end{equation*}
with $C_0$ depending only on suitable critical norms of $\omega_0^\theta.$
\smallbreak
 We are going to see that the same occurs for all  third order partial derivatives of $\omg^\tht$. 

\subsubsection*{Bounding   $\d^3_{zzr}\omega^\theta$} 
Differentiating  \eqref{eq:axiEuler-wo} once with respect to $r$ gives
\begin{equation}\label{eq:axiEuler-wo1}
D_t\d_r\omega^\theta= v\d_r\omega^\theta+\omega^\theta\d_r v-\d_ru\cdot\nabla\omega^\theta .
\end{equation}
Therefore, 
\begin{multline*}
D_t \d^3_{zzr}\omega^\theta= v\d^3_{rzz}\omega^\theta+2\d_zv\d^2_{rz}\omega^\theta+\d^2_{zz}v\d_r\omega^\theta
+\d^2_{zz}\omega^\theta\d_r v  
+2\d_z\omega^\theta\d^2_{rz} v  \\+\omega^\theta\d^3_{rzz} v  
- \d^2_{zz}u\cdot\nabla\d_r\omega^\theta-2\d_zu\cdot\nabla\d^2_{rz}\omega^\theta-\d^3_{rzz}u\cdot\nabla\omega^\theta
-2\d^2_{rz}u\cdot\nabla \d_z\omega^\theta-\d_r u\cdot\nabla^2_{zz}\omega^\theta.
\end{multline*}
We have
\begin{align*}
\|v\d^3_{rzz}\omega^\theta\|_{L^1}&\lesssim\|v\|_{L^\infty}\|\d^3_{rzz}\omega^\theta\|_{L^1} \\
\|\d_zv\d^2_{rz}\omega^\theta\|_{L^1}&\lesssim \|{\d_zv}\|_{L^3}\|\d^2_{rz}\omega^\theta\|_{L^{3/2}} \\
\|\d^2_{zz}v\d_r\omega^\theta\|_{L^1}&\lesssim\|\d^2_{zz}u^r\|_{L^3}\|r^{-1}\d_r\omega^\theta\|_{L^{3/2}} \\
\|\d^2_{zz}\omega^\theta\d_r v \|_{L^1}&\lesssim\|\d^2_{zz}\omega^\theta\|_{L^{3/2}}\|\d_r v \|_{L^3}\\ 
\|\d_z\omega^\theta\d^2_{rz} v \|_{L^1}&\lesssim  \|\d_z\omega^\theta\|_{L^3}\|\d^2_{rz} v \|_{L^{3/2}}\\ 
\|\d^2_{zz}u\cdot\nabla\d_r\omega^\theta\|_{L^1}&\lesssim\|\d^2_{zz}u\|_{L^3}\|\nabla\d_r\omega^\theta\|_{L^{3/2}} \\
\|\d_zu\cdot\nabla\d^2_{rz}\omega^\theta\|_{L^1}&\lesssim \|\d_zu\|_{L^\infty}\|\nabla\d^2_{rz}\omega^\theta\|_{L^1}\\
\|\d^3_{rzz}u\cdot\nabla\omega^\theta\|_{L^1}&\lesssim \|\d^3_{rzz}u\|_{L^{3/2}}\|\nabla\omega^\theta\|_{L^3}\\
\|\d^2_{rz}u\cdot\nabla \d_z\omega^\theta\|_{L^1}&\lesssim \|\d^2_{rz}u\|_{L^3}\|\nabla \d_z\omega^\theta\|_{L^{3/2}} \\
\|\d_r u\cdot\nabla\d^2_{zz}\omega^\theta\|_{L^1}&\lesssim \|\d_r u\|_{L^\infty}\|\nabla\d^2_{zz}\omega^\theta\|_{L^1}.
\end{align*}
The only term that cannot be directly bounded from our previous computations is $\omega^\theta\d^3_{rzz} v .$ To handle it, we observe that
$$\omega^\theta\d^3_{rzz} v =\alpha \d^3_{rzz}u^r-\alpha\d^2_{zz}v,$$
which stems from the identity\footnote{A consequence of the commutator formula $[r,\d^k_{r}] = - k\d_{r}^{k-1}.$}
\begin{equation}\label{eq:com}
	r \partial^k_{r}(r^{-1}f) =  \partial_r^k f - k \d_{r}^{k-1}(r^{-1}f).\end{equation}
	Hence we have 
$$\|\omega^\theta\d^3_{rzz} v \|_{L^1} \lesssim\|\alpha \|_{L^3}\bigl(\|\d^3_{rzz}u^r\|_{L^{3/2}}
+\|\d^2_{zz}v\|_{L^{3/2}}\bigr),$$
which grows at most double exponentially in time. 

\subsubsection*{Bounding  $\d^3_{zrr}\omega^\theta$ }
Differentiating  \eqref{eq:axiEuler-wo1} once with respect to $r$ gives
\begin{equation}\label{eq:axiEuler-wo2}
D_t\d^2_{rr}\omega^\theta= -\d^2_{rr}u\cdot\nabla\omega^\theta-2\d_ru\cdot\nabla\d_r\omega^\theta
+\omega^\theta\d^2_{rr}v+2\d_r\omega^\theta\d_rv+ v\d^2_{rr}\omega^\theta.\end{equation}
Then, differentiating once with respect to $z,$ we discover that
\begin{multline*}
D_t\d^3_{zrr}\omega^\theta= -\d^3_{rrz}u\cdot\nabla\omega^\theta -\d^2_{rr}u\cdot\nabla\d_z\omega^\theta
-2\d^2_{rz}u\cdot\nabla\d_r\omega^\theta-2\d_ru\cdot\nabla\d^2_{rz}\omega^\theta\\
+\omega^\theta\d^3_{rrz}v+\d_z\omega^\theta\d^2_{rr}v
+2\d^2_{rz}\omega^\theta\d_rv+2\d_r\omega^\theta\d^2_{rz}v
+ \d_zv\d^2_{rr}\omega^\theta+ v\d^3_{rrz}\omega^\theta.
\end{multline*}
We have
\begin{align*}
\|\d^3_{rrz}u\cdot\nabla\omega^\theta\|_{L^1}&\lesssim \|\d^3_{rrz}u\|_{L^{3/2}}\|\nabla\omega^\theta\|_{L^3}\\
\|\d^2_{rr}u\cdot\nabla\d_z\omega^\theta\|_{L^1}&\lesssim\|\d^2_{rr}u\|_{L^3}\|\nabla\d_z\omega^\theta\|_{L^{3/2}} \\ 
\|\d^2_{rz}u\cdot\nabla\d_r\omega^\theta\|_{L^1}&\lesssim\|\d^2_{rz}u\|_{L^3}\nabla\d_r\omega^\theta\|_{L^{3/2}} \\
\|\d_ru\cdot\nabla\d^2_{rz}\omega^\theta\|_{L^1}&\lesssim \|\d_ru\|_{L^\infty}\|\nabla\d^2_{rz}\omega^\theta\|_{L^1}\\
\|\d_z\omega^\theta\d^2_{rr}v\|_{L^1}&\lesssim\|\d_z\omega^\theta\|_{L^3}\|\d^2_{rr}v\|_{L^{3/2}} \\
\|\d^2_{rz}\omega^\theta\d_rv\|_{L^1}&\lesssim\|\d^2_{rz}\omega^\theta\|_{L^{3/2}}\|\d_rv\|_{L^3} \\
\|\d_r\omega^\theta\d^2_{rz}v\|_{L^1}&\lesssim\|\d_r\omega^\theta\|_{L^3}\d^2_{rz}v\|_{L^{3/2}} \\
\| \d_zv\d^2_{rr}\omega^\theta\|_{L^1}&\lesssim\| \d_zv\|_{L^3}\|\d^2_{rr}\omega^\theta\|_{L^{3/2}} \\
\|v\d^3_{rrz}\omega^\theta\|_{L^1}&\lesssim \|v\|_{L^\infty}\|\d^3_{rrz}\omega^\theta\|_{L^1}.
\end{align*}
Only $\omega^\theta\d^3_{rrz}v$ is not straightforward. To handle it, we note that, owing to \eqref{eq:com}, 
$$
\omega^\theta\d^3_{rrz}v=\alpha \d^3_{rrz} u^r -2 \alpha\d^2_{rz}  v ,$$
whence
$$\| \omega^\theta\d^3_{rrz}v\|_{L^1}\leq \|\alpha\|_{L^3} \bigl(\|\d^3_{rrz} u^r\|_{L^{3/2}}+2\|\d^2_{rz} v \|_{L^{3/2}}\bigr)\cdotp$$

\subsubsection*{Bounding  $\d^3_{rrr}\omega^\theta$} 
Differentiating  \eqref{eq:axiEuler-wo2} once with respect to $r$ gives
\begin{multline}\label{eq:axiEuler-wo3}
D_t\d^2_{rrr}\omega^\theta=  -\d^3_{rrr}u\cdot\nabla\omega^\theta-3\d^2_{rr}u\cdot\nabla\d_r\omega^\theta
-3\d_ru\cdot\nabla\d^2_{rr}\omega^\theta\\
+\omega^\theta\d^3_{rrr}v+3\d_r\omega^\theta\d^2_{rr}v+3\d^2_{rr}\omega^\theta\d_r v 
+ v\d^3_{rrr}\omega^\theta.
\end{multline}
The most  tricky term is  $	\omg^{\tht}\d^3_{rrr}v.$ However, using the formula \eqref{eq:com}
	 we have \begin{equation*}
		\begin{split}
			\omg^{\tht}\d^3_{rrr}v = \alpha\d^3_{rrr}u^{r}  - 3\alpha\d_{rr}^2v.
		\end{split}
	\end{equation*} After this rewriting,  we discover that
	\begin{equation*}
			\nrm{ \omg^{\tht}\d_{rrr}^3v }_{L^{1}} \le \left(  \nrm{ \d_{rrr} u^{r} }_{L^{3/2}} +3 \nrm{\d^2_{rr}v}_{L^{3/2}}  \right) \nrm{\alpha}_{L^{3}}.  
			\end{equation*}
To bound the other terms of the right-hand side of \eqref{eq:axiEuler-wo3}, it is business as usual:
\begin{align*}
\|\d^3_{rrr}u\cdot\nabla\omega^\theta\|_{L^1}&\lesssim \|\d^3_ru\|_{L^{3/2}}\|\nabla\omega^\theta\|_{L^3} \\
\|\d^2_{rr}u\cdot\nabla\d_r\omega^\theta\|_{L^1}&\lesssim \|\d^2_{rr}u\|_{L^3}\|\d_r\omega^\theta\|_{L^{3/2}} \\
\|\d_ru\cdot\nabla\d^2_{rr}\omega^\theta\|_{L^1}&\lesssim \|\d_ru\|_{L^\infty} \|\nabla\d^2_{rr}\omega^\theta\|_{L^1}\\
\|\d_r\omega^\theta\d^2_{rr}v\|_{L^1}&\lesssim \|\d_r\omega^\theta\|_{L^3}\|\d^2_{rr}v\|_{L^{3/2}}\\
\|\d^2_{rr}\omega^\theta\d_r v \|_{L^1}&\lesssim \|\d^2_{rr}\omega\|_{L^{3/2}}\|\d_r v \|_{L^3} \\
\|v\d^3_{rrr}\omega^\theta\|_{L^1}&\lesssim  \|v\|_{L^\infty}\|\d^3_{rrr}\omega^\theta\|_{L^1}.
\end{align*}
Putting all the previous computations together, we conclude that \begin{equation*}
	\begin{split}
		\frac{d}{dt} \nrm{ \nb^{3} \omega^\theta }_{L^{1}} \le C_0 \exp(C_0t) \nrm{ \nb^{3}\omega^\theta }_{L^{1}}  + C_0 \exp( C_0 \exp (C_0t)), 
	\end{split}
\end{equation*} which concludes at most double exponential growth of $ \nrm{ \nb^{3} \omega^\theta }_{L^{1}} $ with respect to time.

\subsection*{Acknowledgments}
RD was partially supported by Institut Universitaire de France. IJ was supported by the NRF grant from the Korea government (MSIT), No. 2022R1C1C1011051, RS-2024-00406821, the Asian Young Scientist Fellowship, and an individual grant from KIAS. We thank Tarek Elgindi for helpful discussions.


\appendix 

\section{Lorentz and Besov spaces}

For the reader's convenience, we recall a few properties of Lorentz and Besov spaces. 

 
 To any measurable function $f:\R^d\to\R,$ we associate its distribution function  $\Lambda_f:[0,\infty)\to[0,\infty)$ defined by
 $$\Lambda_{f}(t) = \left| \left\{ x \, : \, |f(x)| > t \right\} \right|.$$
 We recall that the $L^p$ norms ($p\in[1,\infty)$) may be computed by means of the distribution function as follows:
 $$ \|f\|_{L^p}^p= p\int_0^\infty \Bigl(t\bigl(\Lambda^{\frac1p}_f(t)\bigr)\Bigr)^p\frac{dt}t\cdotp$$ 
 The Lorentz spaces $L^{p,1}$ are defined by imposing the following stronger condition:
 $$ \|f\|_{L^{p,1}}:= p\int_0^\infty \Lambda^{\frac1p}_f(t)\,dt<\infty.$$  
Lorentz spaces $L^{p,1}$ with $1<p<\infty$ may be equivalently defined from Lebesgue spaces by real interpolation 
as follows: 
 \begin{equation}\label{eq:interpo}
 L^{p,1}=[L^{p_1},L^{p_2}]_{\theta,1},\ \ \hbox{if }\ \ 1\leq p_1<p < p_2\leq \infty\andf \frac1p=\frac{1-\theta}{p_1}+\frac\theta{p_2}\cdotp
 \end{equation}

 \begin{prop} \label{l:Lorentz1} Assume that  $d\geq2$ and denote $d'=d/(d-1).$ Then,  we have the continuous embedding $W^{1,1}(\R^d)\hookrightarrow
 L^{d',1}(\R^d)$ with the inequality 
 $$ \|z\|_{L^{d',1}}\leq C_d\|\nabla z\|_{L^1}.$$
 \end{prop}
 \begin{proof} Arguing by density, it suffices to prove the result for $C^1$ compactly supported functions.
 Then,  the inequality stems from  coarea formula and the isoperimetric inequality. 
 More precisely, coarea formula reads 
   \[
   \int_{\mathbb{R}^d} |\nabla u| \, dx = \int_{-\infty}^\infty \mathcal{H}_{d-1}(\{u = \lambda\}) \, d\lambda,
   \]
   where  \(\mathcal{H}_{d-1}\) stands for the $d-1$-dimensional  Hausdorff measure on $\R^d.$
   \medbreak
   Next, the isoperimetric inequality stipulates that, for any bounded Borel set $E$ of $\R^d,$  
     \[ d  B_d^{1/d} |E|^{1/d'} \leq  \mathcal{H}_{d-1}(\partial E).\]
 Consequently, taking  \(E= \{u > \lambda\}\) with $\lambda$ describing $\R$ gives
   \[   |\{u > \lambda\}|^{1/d'} \leq C_d \, \mathcal{H}_{d-1}(\{u = \lambda\}).\]
   Then, integrating from $\lambda=-\infty$ to $\lambda=\infty,$ we discover that
   \[\int_{-\infty}^\infty |\{u > \lambda\}|^{1/d'} \, d\lambda \leq C_d\int_{-\infty}^\infty    \mathcal{H}^1(\{u = \lambda\}) \, d\lambda=C_d\|\nabla u\|_{L^1}. \]
   The left-hand side is equivalent to the usual 
   Lorentz (quasi)-norm, whence the result. 
 \end{proof}
 
 \begin{prop} \label{l:Lorentz2}  In any dimension $d\geq1$ and  for all Lebesgue exponents $1<p<q\leq\infty,$ we have the continuous embedding 
 $$L^{p,1}(\R^d)\hookrightarrow \dot B^{\frac dq-\frac dp}_{q,1}(\R^d). $$
 \end{prop}
 \begin{proof}
 Let us fix some small enough $\ep>0$ and define the Lebesgue exponents $p^\pm$ from the relations
 $$ \frac 1{p^+}=\frac 1p -\frac 1\ep\andf  \frac 1{p^-}=\frac 1p +\frac 1\ep\cdotp$$
 Then, we have the classical embedding
 $$
 L^{p^-}\hookrightarrow \dot B^{\frac dq-\frac d{p^-}}_{q,\infty}\andf L^{p^+}\hookrightarrow \dot B^{\frac dq-\frac d{p^+}}_{q,\infty}
 $$
 whence, according to \eqref{eq:interpo}, 
 $$ L^{p,1}=[L^{p^-},L^{p^+}]_{\frac12,1}\hookrightarrow[ \dot B^{\frac dq-\frac d{p^-}}_{q,\infty}, \dot B^{\frac dq-\frac d{p^+}}_{q,\infty}]_{\frac12,1}
=\dot B^{\frac dq-\frac dp}_{q,1}.$$
This is exactly what is wanted.  
 \end{proof}
 \begin{prop}\label{l:BS}  For all  $1<p<\infty,$ the  operator {$\nb^2\Delta^{-1}$} maps $L^p$ to itself and $L^{p,1}$ to itself.
 \end{prop}
 \begin{proof} Being a combination of Riesz transforms, {$\nb^2\Delta^{-1}$} maps all Lebesgue space
 $L^p$ with $1<p<\infty$ to itself. The result for general Lorentz spaces follows from the real interpolation 
 theory, see \cite{BL}. 
 \end{proof}

\section{Equivalence of norms for axisymmetric solutions}\label{sec:equivalence}

Here we consider smooth divergence-free velocity fields $\bfu = u^{r}(t,r,z) \bfe^{r} + u^{z}(t,r,z) \bfe^{z},$
with corresponding vorticity $\bfomg = \omg^{\tht}(t,r,z)\bfe^{\tht}.$ We want to characterize various
 Lebesgue or Lorentz norms of derivatives of $\bfu$ and $\bfomg$ in terms of derivatives of  their components with respect to $r$ and to $z.$ In what follows, we use $\nb$ without subscripts to denote the usual gradient in Cartesian coordinates, while $\nb_{r,z} = (\rd_r, \rd_z)$. 
	\begin{lem}\label{l:nablaomega}
Let $X$ denote either the Lebesgue space $L^p$ with $1\leq p\leq\infty,$
or the Lorentz space $L^{p,1}$ with $1<p<\infty.$ 
		For $\bfomg = \omg^{\tht}(t,r,z)\bfe^{\tht},$ we have \begin{equation*}
			\begin{split}
				\nrm{\nb\bfomg}_{X} &\approx \nrm{\nb_{r,z}\omg^{\tht}}_{X} + \nrm{r^{-1}\omg^{\tht}}_{X},  \\ \|\nabla^2 \bfomg\|_{X} &\approx \|\nb_{r,z}^2 \omega^\theta\|_{X} + \nrm{ \nb_{r,z} (r^{-1}\omg^{\tht}) }_{X}, \\
				\nrm{\nb^3\bfomg}_{X} &\approx \nrm{\nb_{r,z}^{3} \omg^{\tht}}_{X} + \nrm{ \nb_{r,z}^{2} (r^{-1}\omg^{\tht}) }_{X} + \nrm{ r^{-1}\d_{r}(r^{-1}\omg^{\tht}) }_{X}. 
			\end{split}
		\end{equation*} 
	\end{lem}
	\begin{proof}
		The lemma follows from direct \textit{pointwise} identities relating gradients in Cartesian coordinates and those in cylindrical: we have $$|\nabla \bfomg(x)|^2 =  |\partial_r \omega^\theta|^2 + \left|\frac{\omega^\theta}{r}\right|^2 + |\partial_z \omega^\theta|^2,  $$ $$|\nabla^2 \bfomg|^2 = (\partial_r^2 \omega^\theta)^2 + (\partial_z^2 \omega^\theta)^2 + 2(\partial_r \partial_z \omega^\theta)^2 + 3\left(\frac{\partial_r \omega^\theta}{r} - \frac{\omega^\theta}{r^2}\right)^2 + 2\left(\frac{\partial_z \omega^\theta}{r}\right)^2,$$ 
		\begin{multline*}
				|\nabla^3 \bfomg|^2 = \ (\partial_r^3 \omega^\theta)^2 + 3(\partial_r^2 \partial_z \omega^\theta)^2 + 3(\partial_r \partial_z^2 \omega^\theta)^2 + (\partial_z^3 \omega^\theta)^2  + 6 \left(\frac{\partial_r^2 \omega^\theta}{r} - \frac{2\partial_r \omega^\theta}{r^2} + \frac{2\omega^\theta}{r^3}\right)^2 \\
				+ 9 \left(\frac{\partial_r \partial_z \omega^\theta}{r} - \frac{\partial_z \omega^\theta}{r^2}\right)^2  + 3 \left(\frac{\partial_z^2 \omega^\theta}{r}\right)^2  + 9 \left(\frac{\partial_r \omega^\theta}{r^2} - \frac{\omega^\theta}{r^3}\right)^2 .
				 \mbox{\qedhere} 
				\end{multline*}
	\end{proof} 
	Similar computations lead to the following statement: 
\begin{lem}\label{l:nablau} Let $X$ denote either  $L^p$ with $1\leq p\leq\infty,$
or $L^{p,1}$ with $1<p<\infty.$ 
For	$\bfu = u^{r}(t,r,z) \bfe^{r} + u^{z}(t,r,z) \bfe^{z},$ we have the equivalences 
\begin{align*}
	 \|\nabla \bfu\|_{X} &\approx \| \nb_{r,z} u\|_{X} + \|r^{-1} u^{r} \|_{X} , \\
	\|\nabla^2 \mathbf{u}\|_{X} &\approx   \nrm{\nb^2_{r,z} u  }_{X}  + \nrm{ \nb_{r,z}(r^{-1} u^{r})}_{X}  + \nrm{ r^{-1} \d_{r} u^{z}}_{X}  , \\
			\|\nabla^3 \mathbf{u}\|_{X} &\approx   \|\nabla^3_{r,z}  {u}\|_{X}  + \|\nabla^2_{r,z}(r^{-1}u^r)\|_{X} + \|r^{-1}\partial_r(r^{-1}u^r)\|_{X}  + \|\nabla_{r,z}(r^{-1}\partial_r u^z)\|_{X} 
		\end{align*} 
		where we write $u = (u^r, u^z)$. 
	\end{lem}
Then, combining with Proposition \ref{l:BS}, we conclude that:	
	\begin{corollary} \label{cor:Lp}
		For  $1 < p < \infty$ and $X\in\{L^p,L^{p,1}\},$  we have 
		\begin{equation*}
				\| \nb_{r,z} u\|_{X} + \|r^{-1} u^{r} \|_{X}  \le C \nrm{\omg^{\tht}}_{X}, 
					\end{equation*}\begin{equation*}
					\nrm{\nb^2_{r,z} u  }_{X}  + \nrm{ \nb_{r,z}(r^{-1} u^{r})}_{X}  + \nrm{ r^{-1} \d_{r} u^{z}}_{X}  \le C \left( \nrm{\nb_{r,z}\omg^{\tht}}_{X} + \nrm{r^{-1}\omg^{\tht}}_{X}  \right)
				\end{equation*} \begin{multline*}
			\|\nabla^3_{r,z}  {u}\|_{X}  + \|\nabla^2_{r,z}(r^{-1}u^r)\|_{X} + \|r^{-1}\partial_r(r^{-1}u^r)\|_{X}  + \|\nabla_{r,z}(r^{-1}\partial_r u^z)\|_{X} \\\le C \left(\|\nb_{r,z}^2 \omega^\theta\|_{X} + \nrm{ \nb_{r,z} (r^{-1}\omg^{\tht}) }_{X} \right)
				\end{multline*} with some $C > 0$ depending only on $p$.
				 	\end{corollary}


  \end{document}